\documentclass{amsart}
\usepackage{enumerate}
\newtheorem{thm}{Theorem}[section]

\newtheorem{lemma}[thm]{Lemma}

\begin{document}

%\begin{frontmatter}

\title{Cubic Thue Inequalities  With Positive Discriminant}

\author{Shabnam Akhtari}
\address{ Fenton Hall\\
University of Oregon\\
Eugene, OR 97403-1222 USA}

 \email {akhtari@uoregon.edu} 
 
\begin{abstract}
 We will give an explicit upper bound for the number of solutions to cubic inequality $|F(x , y)| \leq h$, where $F(x , y)$ is a cubic binary form with integer coefficients and positive discriminant $D$.  Our upper bound is independent of $h$, provided that $h$ is smaller than $D^{1/4}$.
\end{abstract}
%\begin{document}
%\maketitle

\maketitle

\section{Introduction}\label{6s1}

Let $F(x , y)$ be a binary form with integral coefficients. Also assume that the maximal number of pairwise non-proportional linear forms over $\mathbb{C}$ dividing $F$ is at least $3$.  For positive integer $h$ consider the following   \emph{Thue inequality}
\begin{equation}\label{ineq}
0 < \left|F(x , y)\right| \leq h.
\end{equation}
This inequality can be considered as a finite number of  \emph{Thue equations}
\begin{equation}\label{Thue}
F(x , y) = m,
\end{equation}
 where $0 < |m| \leq h$. By a well-known result of Thue \cite{Thu},  the equation (\ref{Thue}) has only finitely many solutions in integers $x$ and $y$.  From this, one can easily deduce that  inequality (\ref{ineq}) has only finitely many solutions in integers $x$ and $y$.

 In this manuscript, we will study the inequality $| F(x , y) | \leq h$
where $h$ is a positive  integer smaller than $D^{1/4}$. Our goal is to give an upper bound for the number of   solutions to the above inequality.  For a treatment of cubic Thue inequalities with negative discriminant, we refer the reader to a work of Wakabayashi  \cite{Wak}, where  Pad\'e approximations and Ricket's integrals are used.  In general,  upper bounds for the number of solutions to a Thue equation $F(x , y) =m$, depend on the number of prime factors of $m$ (see \cite{Mah23}, \cite{Bom} and \cite{Ste}). The problem of counting the number of solutions of Thue inequalities $F(x , y) \leq h$  and obtaining upper bounds independent of the value of $h$ for ``small'' integers $h$ has been studied by many mathematicians. In 1929, Siegel \cite{Sie},  by means of an approximation method in which hypergeometric functions are used, showed that the number of solutions of the cubic inequality
$$
1\leq F(x , y) \leq h
$$
in integers $x$, $y$ with $\gcd(x,y)= 1$,  $y> 0$ or $(x,y) = (1,0)$,  is at most $18$ if $D > 4.3^{107} h^{30}$. In 1949, A.E. Gel'man showed, by refining Siegel's estimates, that $18$ can be replaced by $10$ if one assumes that $D > ch^{36}$, where $c$ is an absolute constant (we refer the reader to Chapter 5 of \cite{Del} for a proof).
The following are  our main theorems.
\begin{thm}\label{newmain}
Let  $F(x , y)$ be a binary cubic form with positive discriminant $D$.   Suppose $h$ is an integer and satisfies  $h = \frac{1}{2\pi} (3D)^{1/4 -\epsilon}$ for a positive value $\epsilon <1/4$. Then the inequality
  \begin{equation*}\label{mainineq}
\left|F(x , y)\right| \leq h
\end{equation*}
has at most $ 9  +\frac{  \log \left(\frac{3}{8\epsilon} +  \frac{1}{2}\right)}{\log 2}  $ solutions in  coprime integers $x$ and $y$ with $y \neq 0$.
\end{thm}

The upper bound in Theorem \ref{newmain} gets larger as $\epsilon$ approaches $0$. It turns out that we can get a better upper bound if we use the value of discriminant in our upper bound; in other words, if we are willing to have more dependence on our form $F$ in the upper bound. 
\begin{thm}\label{main}
  Let  $F(x , y)$ be a binary cubic form with positive discriminant $D$. Suppose $h$ is an integer satisfying  $h = \frac{1}{2\pi} (3D)^{1/4 -\epsilon}$ with  $0 < \epsilon < 1/4$. Then the inequality
 \begin{equation*}
\left|F(x , y)\right| \leq h
\end{equation*}
  has at most $ 12 +\frac{3}{\log 2}  \log \frac{3}{8 \left( \left(\frac{3}{2}\right)\epsilon +  \frac{ \log 2}{\log (3D)} \right)} $ solutions in  coprime integers $x$ and $y$ with $y \neq 0$.
\end{thm}

Note that our bound in the above theorem can be seen as an absolute bound, as the number of conjugacy classes of cubic binary forms with bounded discriminant is finite.

It is worth to mention here that in order  to have our upper bounds independent of $h$,  we must take $h$ smaller than $D^{1/4}$, and therefore, in some sense the dependence of $h$ on $D$ is sharp in Theorems  \ref{newmain} and \ref{main}. A cubic form of discriminant $D$ looks generically like
\begin{equation}\label{gen}
G(x , y) = l_{0}D^{1/4} x^3 + l_{1}D^{1/4}x^2y + l_{2} D^{1/4} xy^2 + l_{3} D^{1/4} y^3,
\end{equation}
where $l_{i}$ are small numbers (see the formula for discriminant in (\ref{disc}) ). Now consider the inequality 
$$
G(x , y) \leq k.
$$
Assume  $k > l_{0} D^{1/4}$ and  write $k = v l_{0} D^{1/4}$, where $v >1$. Then for any integer $X$ with $|X| \leq v$, the pair $(X, 0)$ is a solution to (\ref{gen}). Therefore the number of solutions depends on $v$ and consequently on $k$. To avoid only having solutions of the form $(x , 0)$, we can let $SL_{2}(\mathbb{Z})$ act on the form $G$ and construct non-trivial solutions; solutions $(x , y)$ with $y \neq 0$ (see Section \ref{6s2}  for details on equivalent forms). It is well-known that a binary cubic form with positive discriminant $D$ is equivalent to a cubic form, namely a reduced form,  with leading coefficient that is bounded in absolute value by $ D^{1/4}$. Reduced forms are defined in Section \ref{6s2}.

%%%%%%%%%%%%%%%%%%%%%%%%%%%%%%%%%%%%%%%%%%%%%%%%%%%%%%%%%%%%%%%%
The problem in hand has been studied by Evertse and Gy\H{o}ry (see \cite{EG1} and \cite{EG16})  in a  general set up; they have considered the Thue inequality
$$0 < |F(x , y)| \leq h$$ for binary forms of arbitrary degree $n\geq 3$.
Define, for $ 3 \leq n < 400$
$$
\left( N(n) , \delta(n) \right) = 
\left( 6 n 7^{\binom{n}{3}},\,  \frac{5}{6} n (n-1) \right) $$
and for $n > 400$
$$
\left( N(n) , \delta(n) \right) = \left( 6n , \, 120(n -1)\right). 
$$
They prove that if
$$
|D| > h^{\delta{n}} \exp (80 n (n -1))
$$
then the number of solutions to $0 < |F(x , y)| \leq h$ in co-prime integers $x$ and $y$, is at most $N(n)$.
Further, Gy\H{o}ry \cite{Gyo1} showed for binary forms $F$ of degree $n \geq 3$, that if $0 <a < 1$ and 
$$
\left| D \right| \geq n^n (3.5^{n} h^2)^{\left(2(n-1)/(1-a)\right)},
$$
then
the number of solutions to $0 < |F(x , y)| \leq h$ in co-prime integers $x$ and $y$ is
at most $$25n + (n+2) \left(\frac{2}{a} + \frac{1}{4}\right),$$
and if $F$ is reducible then at most 
$$5n + (n+2) \left(\frac{2}{a} + \frac{1}{4}\right).$$
Here we are able to improve these results for the particular case of $n=3$.

The cubic equation 
\begin{equation}\label{61.2}
|F(x , y)| = 1
\end{equation}
 has been very well studied.  Evertse \cite{Ev}  refined the techniques used in \cite{Sie} to show for  irreducible cubic form $F$  of positive discriminant, that equation
(\ref{61.2})
 has at most $12$ solutions in integers $x$ and $y$. Later, Bennett \cite{Ben} improved this upper bound to $10$.
Here we will  appeal to methods that deal with the equation $|F(x , y)| = 1$, particularly those from \cite{Ev}, to prove similar upper bounds for the number of solutions to (\ref{ineq}), when $h = \frac{1}{2\pi} (3D)^{1/4 -\epsilon}$ for a positive value of $\epsilon$.

Throughout this manuscript, by a solution $(x , y)$, we mean $x, y \in \mathbb{Z}$, $y \neq 0$ and $\gcd(x , y) =1$.

%-------------------------------------------------------------------------------------
\section{Invariants and  Covariants  of Binary Cubic Forms }\label{6s2}
%------------------------------------------------------------------------------------------

Let
$F = ax^{3} + bx^{2}y + cxy^{2} + dy^{3}$
be an irreducible binary cubic form. The discriminant of $F$ is 
\begin{equation}\label{disc}
D = 18 a b c d + b^{2} c^{2} - 27 a^{2} d^{2} - 4 a c^{3} - 4 b^{3} d = a^{4} \prod_{i , j} (\alpha_{i} - \alpha_{j})^{2},
\end{equation}
where $\alpha_{1}$, $\alpha_{2}$ and $\alpha_{3}$ are the roots of polynomial $F(x , 1)$.

For the cubic form $F$, we define  an associated quadratic form, the Hessian $H = H_{F}$, and a cubic form $G = G_{F}$, by 
\begin{eqnarray*}
H(x , y) & = & - \frac{1}{4} \left(\frac{\delta^{2}F}{\delta x^{2}}\frac{\delta^{2}F}{\delta y^{2}} -\left( \frac{\delta ^{2}F}{\delta x \delta y}\right)^{2}\right) \\ \nonumber
& = &  A x^{2} + B x y + C y^{2} 
\end{eqnarray*}
and
$$
G(x,y) = \frac{\delta F}{\delta x} \frac{\delta H}{\delta y} - \frac{\delta F}{\delta y} \frac{\delta H}{\delta x}.
$$
These forms satisfy a covariance property; i.e.
$$
H_{F \circ \gamma } = H_{F} \circ \gamma \qquad  \textrm{and} \qquad  G_{F \circ \gamma } = G_{F} \circ \gamma
$$
for all $\gamma \in GL_{2}(\mathbb{Z})$.

We call forms $F_{1}$ and $F_{2}$ equivalent if they are equivalent under $GL_{2} (\mathbb{Z})$-action; i.e. if there exist integers $a_{1}$ , $a_{2}$ , $a_{3}$ and $a_{4}$ such that 
$$
F_{1}(a_{1} x + a_{2}y , a_{3}x + a_{4}y) = F_{2} (x , y) 
$$
for all $x$, $y$ , where $a_{1} a_{4} - a_{2} a_{3} = \pm 1$.

We denote by $N_{F}$ the number of solutions in coprime integers $x$ and $y$ of inequality  $F(x , y) \leq h$. If $F_{1}$ and $F_{2}$ are equivalent, then $N_{F_{1}} = N_{F_{2}}$  and  $D_{F_{1}} = D_{F_{2}}$.

For
$F(x,y) = ax^{3} + bx^{2}y + cxy^{2} + dy^{3}$ with discriminant $D$, it follows by routine calculation that 
$$A = b^{2} -3 a c  ,\ B = b c - 9a d ,\   C = c^{2} -3 b d $$
and 
$$ B^{2} - 4 AC = -3 D .$$  
Note that  if $F$ has positive discriminant  then $H$ is positive definite.

Further, these forms are related to $F (x , y)$ via the identity
\begin{equation}\label{65.1}
4 H(x , y)^{3} = G (x , y) ^{2} + 27 D F(x , y) ^{2}.
\end{equation}

Binary cubic form $F$ is called \emph{reduced}  if the Hessian  
$$H(x , y) = A x^{2} + B x y + C y^{2}$$ 
of $F$
satisfies
$$
C \geq A \geq |B|.
$$
It is a basic fact (see \cite{Del}) that every cubic form of positive discriminant is equivalent to a reduced form  $F(x , y)$ .

\begin{lemma}\label{6mike}
Let $F$ be an irreducible binary cubic form with positive discriminant $D$ and Hessian $H$.  Assume that $F$ is reduced. For all integer pair $(x_{1} , y_{1})$ with $y_{1} \neq 0$, we have $H(x_{1} , y_{1}) \geq \frac{1}{2}\sqrt{3D} $.
\end{lemma} 
\begin{proof}
This is Lemma 5.1 of \cite{Ben}.
\end{proof}

%------------------------------------------------------------------------------------
\section{Reduction to a Diagonal Form}\label{6s3}
%--------------------------------------------------------------------------------------------

Let $\sqrt{-3D}$ be a fixed choice of the square-root of $-3D$. we will work in the number field $M = \mathbb{Q}(\sqrt{-3D})$. 
By syzygy (\ref{65.1}), one may write 
$$H (x , y)^3 =  U(x, y ) V(x, y)$$
 where 
$$ 
U(x , y) = \frac{G(x , y) + 3\sqrt{-3D} F(x , y)}{2} ,
 $$
$$
 V(x , y) = \frac{G(x , y) - 3\sqrt{-3D} F(x , y)}{2}.
 $$ 
Then $ U$ and $V$ are cubic forms with coefficients belonging to $M$ such that corresponding coefficients of $U$ and $V$ are complex conjugates. Since $F$ must be also irreducible over $M$,  $U$ and $V$ do not have factors in common. It follows that $U(x , y)$ and $V(x , y)$ are cubes of linear forms over $M$, say $\xi (x , y)$ and $\eta (x,y )$.

Note that $\xi (x , y) \eta (x , y) $ must be a quadratic form which is cube root of $H(x , y) ^{3}$ and for which the coefficient of $x^{3}$ is a positive real number. Hence we have 
\begin{equation}\label{Fdiag}
\xi (x , y) ^{3} - \eta (x , y)^{3}   =  3\sqrt{-3D} F(x , y),
\end{equation}
\begin{equation}\label{63}
\xi(x , y)^{3} + \eta (x , y)^{3}  = G (x, y),
\end{equation}
\begin{equation}\label{Hdiag}
 \xi(x , y) \eta (x , y) = H (x , y)
 \end{equation}
and 
$$
\frac{\xi (x , y)} {\xi (1 , 0)}  \qquad  \textrm{and} \qquad   \frac{\eta (x , y)}{\eta(1, 0)} \ \in M .
$$
The reason for the last identity is that for any pair of rational integers $x_{0}$ , $y_{0}$, 
$$\xi (x_{0} , y_{0})\,   \qquad  \textrm{and} \qquad  \eta (x_{0} , y_{0})
$$
 are complex conjugates and the discriminant  of $H$ is $ -3D$.

We call a pair of forms $\xi$ and $\eta$ satisfying the above properties a pair of \emph{resolvent} forms.  Note that there are exactly three pairs of resolvent forms, given by 
$$(\xi , \eta), \,  (\omega \xi , \omega ^{2} \eta), \, (\omega^{2}\xi , \omega \xi),
$$
 where $\omega$ is a primitive cube root of unity.

We say that a pair of rational integers $(x , y)$ is related to a pair of resolvent forms if 
\begin{equation}\label{6related}
\left|1 - \frac{\eta(x , y)}{\xi(x , y)}\right| = \min_{0 \leq k \leq 2} \left|\omega ^{k} -\frac{\eta(x , y)}{\xi(x , y)}\right| .
\end{equation}

\begin{thm}\label{T3evertse} (Evertse, 1983)
Let $F$ be a binary cubic form with integral coefficients,  positive discriminant $D$ and  quadratic covariant $H$. Let $k$ be a positive integer. Then the number of solutions of the inequality 
$$
|F(x ,  y)| \leq k
$$
in integers $x, y$ with
$$
H(x , y) \geq \frac{3}{2} (3D)^{1/2} k^3, \qquad  \gcd(x , y) = 1, \qquad y>0 \, \quad \textrm{or}\, \quad  (x, y) = (1, 0)
$$
is at most $9$.
\end{thm}
The following is Lemma 1 of \cite{Ev}.
\begin{lemma}\label{L1evertse}
If $D > 0$ and if $F$ is reduced and irreducible then 
$$
H(x , y) \geq \frac{3}{4} D^{1/2} y^2\, \qquad \textrm{for} \, \quad  x, y \in \mathbb{Z},
$$
$$
H(x , y) \geq \frac{3}{2} D^{1/2} y^2\, \qquad \textrm{for} \, \quad x, y \in \mathbb{Z} \, \quad \textrm{with} \quad |x| \geq |2y|.
$$
\end{lemma}

Evertse \cite{Ev} uses the properties of a reduced form stated in Lemma \ref{L1evertse} to obtain the following:

\begin{thm}\label{T2evertse} (Evertse, 1983)
Let $F$ be a reduced, irreducible binary cubic form of positive discriminant with integral coefficients and let $k$ be a positive integer. Then the inequality 
$$
|F(x , y)| \leq k
$$
has at most nine solutions in integers $x$, $y$ with $\gcd(x , y) =1$ and $y \geq 12^{1/4} k^{3/2}$. 
\end{thm}

Having Theorem  \ref{T3evertse} in hand and  from (\ref{Hdiag}), we conclude that there are at most $9$ solutions $(x , y)$ to inequality $|F |\leq h$  for which 
$$
|\xi(x , y)| \geq \frac{\sqrt{3}}{\sqrt{2}} (3D)^{1/4} h^{3/2}.
$$
All we have to do is to give an upper bound for the number of solutions $(x, y)$ for which 
$$
|\xi(x , y)| < \frac{\sqrt{3}}{\sqrt{2}} (3D)^{1/4} h^{3/2}.
$$

In order to prove our main result, in Section \ref{css}, we will show 
\begin{thm}\label{smallthm}
Let $F$ be a reduced, irreducible binary cubic form of positive discriminant $D$ with integral coefficients.  Suppose that $h = \frac{(3D)^{\frac{1}{4} -\epsilon}}{2\pi}$ with $0 < \epsilon < \frac{1}{4}$. Then the inequality 
$$
0 < |F(x , y)| \leq h
$$
has at most $3\, \frac{ \log \left(\frac{3}{8\epsilon} + \frac{1}{2}\right) }{\log 2}$  solutions in integers $x$, $y$ with $\gcd(x , y) =1$ and $0< y < 12^{1/4} h^{3/2}$. Moreover,  the inequality 
$$
0 < |F(x , y)| \leq h
$$
has at most $3\, \frac{ \log \frac{3}{8 \left( \left(\frac{3}{2}\right)\epsilon +  \frac{ \log 2}{\log 3D} \right)}  }{\log 2}+3$
 solutions in integers $x$, $y$ with $\gcd(x , y) =1$ and $0< y < 12^{1/4} h^{3/2}$. 
\end{thm}

\textbf{Remark}. In the above Theorem, two upper bounds are given. The first one is independent of the discriminant.

%------------------------------------------------------------------------------------------------------------------
\section{Gap Principle}
%-----------------------------------------------------------------------------------------------------------------------------

Let us fix the resolvent forms $(\xi, \eta)$. Our aim is to give an upper bound for the number of primitive solutions $(x , y)$ that are related to $(\xi, \eta)$.
We will first derive an upper bound for
 $$
\left| 1 - \frac{\eta( x , y )}{\xi(x , y)}\right|.
$$
  From our definitions, we have
 $$
\left| 1 - \frac{\eta( x , y )^3}{\xi(x , y)^3}\right| = \frac{3\sqrt{3D}|F(x , y)|}{\xi (x , y)^3} \leq \frac{3\sqrt{3D} h}{\xi(x , y)^3}
$$ 
and will, in consequence of Lemma \ref{6mike}, assume $H(x, y) \geq\frac{\sqrt{3D}}{2}$, whereby
\begin{equation}\label{xiindelta}
|\xi(x , y)| \geq  \frac{1}{\sqrt{2}} (3D)^{1/4}.
\end{equation}
Hence, by (\ref{Fdiag}), we obtain 
$$
\left| 1 - \frac{\eta( x , y )^3}{\xi(x , y)^3}\right|  \leq \frac{6 \sqrt{2}h}{(3D)^{1/4}}.
$$

\textbf{Remark.} From here it is obvious that $h$ has to be bounded in terms of $D$. 

Let us define
 $$
z(x , y) =  1 - \frac{\eta( x , y )^3}{\xi(x , y)^3}.
$$
Since $\eta(x,y)/\xi(x,y)$ has modulus one, we have
 $$|z| < 2.$$

\begin{lemma}\label{cosi}
Suppose that $(x , y)$ is a solution to $|F| \leq h$ and is related to a  pair of resolvent form $(\xi, \eta)$. Let 
$z(x , y) =  1 - \frac{\eta( x , y )^3}{\xi(x , y)^3}$. We have
$$
 \left|1 - \frac{\eta(x , y)}{\xi(x , y)}\right| <\frac{\pi}{6} |z|.
 $$
 Further, if $|z| <1$, we have
 $$
 \left|1 - \frac{\eta(x , y)}{\xi(x , y)}\right| <\frac{\pi}{9} |z|.
 $$
\end{lemma}
\begin{proof}
 Put
 $$
 3\theta = \textrm{arg} \left( \frac{\eta(x , y)^{3}}{\xi(x , y)^{3}}\right).
 $$ 
  We have
  $$|\theta| < \frac{\pi}{3}. $$
Since
  $$
  \sqrt{2 - 2\cos(3\theta)} = |z|,
  $$
  when $|z| < 1$ we have
$$|\theta| < \frac{\pi}{9}.$$
 We have assumed that $(x , y)$ is related to $(\xi, \eta)$. Therefore
 $$ 
\left|1 - \frac{\eta(x,y)}{\xi(x,y)} \right| \leq  |\theta|.
 $$
We conclude that
$$
\left| 1 - \frac{\eta(x,y)}{\xi(x,y)}\right| \leq \frac{1}{3} \frac{|3\theta|}{\sqrt{2 - 2\cos(3\theta)}} 
\left|1 - \frac{\eta(x,y)^{3}}{\xi(x,y)^{3}}\right|. 
$$
By differential calculus $\frac{|3\theta|}{\sqrt{2 - 2\cos(3\theta)}} = \frac{|6\theta|}{\sin{\frac{3}{2}\theta}}$ is increasing on the interval $ 0<|\theta| < \frac{\pi}{3} $ and does not exceed 
$\frac{\pi}{2}$. Therefore
 $$
 \left| 1 - \frac{\eta(x , y)}{\xi(x , y)}\right| <\frac{\pi}{6} |z| ,
 $$
 and similarly, when $|z| <1$,  from the fact that $\frac{|3\theta|}{\sqrt{2 - 2\cos(3\theta)}} < \frac{\pi}{3}$ whenever $ 0<|\theta| < \frac{\pi}{9} $ , we conclude that
 $$
 \left| 1 - \frac{\eta(x , y)}{\xi(x , y)}\right| <\frac{\pi}{9} |z| ,
 $$
as desired.
\end{proof} 

Suppose that we have distinct solutions to $|F(x , y)| \leq h$, related to $(\xi,\eta)$ and indexed by $i$, say $(x_i, y_i)$, with $|\xi(x_{i+1}, y_{i+1})| \geq |\xi(x_i, y_i)|)$. Let us write 
$$\eta_{i}= \eta(x_i, y_i)$$
 and 
 $$\xi_{i} = \xi(x_i, y_i).$$
  Since $\xi(x, y)\eta(x, y) = H(x, y)$ is a quadratic form of discriminant $- 3D$, it follows that
$$
 \xi_2\eta_1 - \xi_1\eta_2 =  \pm\sqrt{-3D} (x_1y_2 - x_2 y_1)
 $$
and, since $(x_1, y_1)$, $(x_2, y_2)$ are distinct solutions to $|F (x, y)| \leq h$, we have
$$
\sqrt{3D} \leq | \xi_2\eta_1 - \xi_1\eta_2 | \leq |\xi_{1}||\xi_{2}| \left(   \left | 1 -\frac{\eta_{1}}{\xi_{1}}  \right|    +  \left | 1 -\frac{\eta_{2}}{\xi_{2}}  \right|   \right).
$$
By Lemma \ref{cosi}, we have
$$
\sqrt{3D} \leq  |\xi_{1}||\xi_{2}| \frac{3 \sqrt{3D}\, h\,  \pi}{6}\left(   \frac{1}{|\xi_{1}|^3} + \frac{1}{|\xi_{2}|^3} \right).
$$
Since we assume that $|\xi_{2}| \geq |\xi_{1}|$,
$$
\sqrt{3D} \leq  |\xi_{2}| \frac{ \sqrt{3D}\, h\,  \pi}{2}\left(   \frac{2}{|\xi_{1}|^2} \right).
$$
Therefore,
\begin{equation}\label{myGP}
  |\xi_{2}|\geq |\xi_{1}|^2 \frac{1}{ h\, \pi}.
\end{equation}
We have $h = \frac{(3D)^{1/4 - \epsilon}}{2\pi}$. From (\ref{xiindelta}), we have $|\xi_{1}| > \frac{(3D)^{1/4}}{\sqrt{2}}$. This implies
$$
  |\xi_{2}|\geq  (3D)^{1/4 + \epsilon}.
$$
Using (\ref{myGP}) for $\xi_{3}$ and $\xi_{2}$, we obtain
$$
  |\xi_{3}|\geq |\xi_{2}|^2 \frac{1}{ h\, \pi} \geq  2 (3D)^{1/4 + 3 \epsilon}.
$$
Applying (\ref{myGP}) once more, we get $ |\xi_{4}| \geq 8 (3D)^{1/4 + 7 \epsilon}$.
Generally for $k> 2$, we have
\begin{equation}\label{GGP}
  |\xi_{k}| \geq 2^{2^{k-2} -1} (3D)^{1/4 + (2^{k-1}- 1) \epsilon}.
\end{equation}

%------------------------------------------------------------------------------------
\section{Counting Small Solutions}\label{css}
%-------------------------------------------------------------------------
We shall use our gap principle established in the previous section, 
to detect an integer $k$ for which 
$$
  |\xi_{k}| > \frac{\sqrt{3}}{\sqrt{2}} (3D)^{1/4} h^{3/2}.
$$
After finding such $k$, we can deduce that there can be at most $k-1$ solutions $(x , y)$ related to a a fixed pair of resolvent form $(\xi, \eta)$ satisfying 
$$
|\xi(x , y)| < \frac{\sqrt{3}}{\sqrt{2}} (3D)^{1/4} h^{3/2}.
$$
To find such $k$, by (\ref{GGP}), it is sufficient for $k$ to be large enough to satisfy
\begin{equation}\label{detectk}
2^{2^{k-2} -1} (3D)^{1/4 + (2^{k-1}- 1) \epsilon} \geq \frac{\sqrt{3}}{\sqrt{2}} (3D)^{1/4} h^{3/2}.
\end{equation}

The following Lemma, together with Theorem \ref{T3evertse}, gives us our main result, Theorem \ref{main}.
\begin{lemma}\label{finallemma}
Let $F$ be an irreducible binary cubic form of positive discriminant $D$ with integral coefficients.  Suppose that $h = \frac{(3D)^{\frac{1}{4} -\epsilon}}{2\pi}$ with $0 < \epsilon < \frac{1}{4}$. Then the inequality 
$$
|F(x , y)| \leq h
$$
has at most $3\, \frac{ \log \left(\frac{3}{8\epsilon} + \frac{1}{2}\right) }{\log 2}$ solutions in integers $x$, $y$ with $\gcd(x , y) =1$   and $0 < H(x , y)< \frac{3}{2} (3D)^{1/2} h^{3}$.
\end{lemma}
\begin{proof}
First we substitute the value of $h = \frac{(3D)^{\frac{1}{4} -\epsilon}}{2\pi}$  in  equation (\ref{detectk}) to get
\begin{equation}\label{ineqinproof}
 (3D)^{(2^{k-1}- \frac{1}{2}) \epsilon - \frac{3}{8}} \geq \frac{\sqrt{3}}{\sqrt{2}\pi}  \frac{1}{2^{2^{k-2}} }.
\end{equation}
Note that if 
$$
\left(2^{k-1}- \frac{1}{2}\right) \epsilon - \frac{3}{8} \geq 0
$$
then the inequality (\ref{ineqinproof}) will hold. 
This means there are at most 
$$
\frac{ \log \left(\frac{3}{8\epsilon} + \frac{1}{2}\right) }{\log 2}
$$
solutions related to each fixed pair of resolvent forms $(\xi , \eta)$.  The statement of the lemma follows immediately.
\end{proof}

\begin{lemma}\label{finallemma2}
Let $F$ be an irreducible binary cubic form of positive discriminant $D$ with integral coefficients.  Suppose that $h = \frac{(3D)^{\frac{1}{4} -\epsilon}}{2\pi}$ with $0 \leq \epsilon < \frac{1}{4}$. Then the inequality 
$$
|F(x , y)| \leq h
$$
has at most $3\, \left( \frac{ \log \frac{3}{8 \left( \left(\frac{3}{2}\right)\epsilon +  \frac{ \log 2}{\log (3D)} \right)}  }{\log 2}\right)+3$ solutions in integers $x$, $y$ with $\gcd(x , y) =1$ and $0 < H(x , y)< \frac{3}{2} (3D)^{1/2} h^{3}$.
\end{lemma}
\begin{proof}

Following the proof of Lemma \ref{finallemma}, we notice that in order to satisfy  inequality (\ref{ineqinproof}), it is sufficient to have
$$
2^{k-2} \log 2 +  \log (3D) \left( (2^{k-1}- \frac{1}{2}) \epsilon - \frac{3}{8}\right) > 0 ;
$$
which means
$$
 \left( (2^{k-1}- \frac{1}{2}) \epsilon - \frac{3}{8}\right)  \log(3D)>  - 2^{k-2} \log 2 .
$$
That is
$$
  \left( (2^{k-1}- \frac{1}{2}) \epsilon - \frac{3}{8}\right) >  -\frac{ 2^{k-2} \log 2}{\log (3D)} .
$$
From here,
$$
2^{k-2}\left( \left(2 - \frac{1}{2^{k-1}}\right)\epsilon +  \frac{ \log 2}{\log (3D)} \right) > \frac{3}{8}.
$$
If $k \geq 2$ then it suffices to have
$$
2^{k-2}\left( \left(\frac{3}{2}\right)\epsilon +  \frac{ \log 2}{\log (3D)} \right) > \frac{3}{8}
$$

$$
k- 2 > \frac{ \log \frac{3}{8 \left( \left(\frac{3}{2}\right)\epsilon +  \frac{ \log 2}{\log (3D)} \right)}  }{\log 2}.
$$
This means there are at most 
$$
\frac{ \log \frac{3}{8 \left( \left(\frac{3}{2}\right)\epsilon +  \frac{ \log 2}{\log (3D)} \right)}  }{\log 2}+1
$$
 solutions with $|\xi(x , y)| < \frac{\sqrt{3}}{\sqrt{2}} (3D)^{1/4} h^{3/2}$  related to each pair of resolvent form $(\xi, \eta)$.
  Since we have $3$ pairs of distinct resolvent forms, our proof is complete.

\end{proof}

Now we can deduce Lemma \ref{smallthm} from  Lemmata \ref{finallemma}, \ref{finallemma2}  and \ref{L1evertse}.

\section{acknowledgements}

I would like to thank the referees for careful reading and helpful suggestions.  This manuscript was written during my visits to Mathematisches Forschungsinstitut Oberwolfach and Max-Planck-Institut f\"{u}r Mathematik Bonn. I would like to thank both institutes for providing me with great work environments.

%-----------------------------------------------------------------------------------------------

\end{document}